\documentclass[11pt,a4paper]{article}

\usepackage{a4wide}
\usepackage{graphicx}
\usepackage{latexsym}
\usepackage{epsfig}
\usepackage{amssymb}
\usepackage{amstext}
\usepackage{amsgen}
\usepackage{amsxtra}
\usepackage{amsgen}
\usepackage{amsthm}

\newtheorem{thm}{Theorem}[section]
\newtheorem{prop}[thm]{Proposition}
\newtheorem{lemma}[thm]{Lemma}
\newtheorem{cor}[thm]{Corollary}

\theoremstyle{definition}
\newtheorem{example}[thm]{Example}
\newtheorem{definition}[thm]{Definition}

\numberwithin{equation}{section}

\newcommand{\R}{\mathbb{R}}

\begin{document}
\title{Bregman Distances in Inverse Problems \\ and Partial Differential Equations}
\author{Martin Burger\thanks{Institut f\"ur Numerische und Angewandte Mathematik, Westf\"alische Wilhelms-Universit\"at (WWU) M\"unster. Einsteinstr. 62, D 48149 M\"unster, Germany. e-mail: martin.burger@wwu.de  } }
\maketitle
\begin{abstract}
The aim of this paper is to provide an overview of recent development related to Bregman distances outside its native areas of optimization and statistics. We discuss approaches in inverse problems and image processing based on Bregman distances, which have evolved to a standard tool in these fields in the last decade. Moreover, we discuss related issues in the analysis and numerical analysis of nonlinear partial differential equations with a variational structure. For such problems Bregman distances appear to be of similar importance, but are currently used only in a quite hidden fashion. We try to work out explicitely the aspects related to Bregman distances, which also lead to novel mathematical questions and may also stimulate further research in these areas.
%

\end{abstract}

\section{Introduction}

Bregman distances for (differentiable) convex functionals, originally introduced in the study of proximal algorithms in \cite{bregman} and named in \cite{censorlent}, are a well established concept in continuous and discrete optimization in finite dimension. A classical example is the celebrated Bregman projection algorithm for finding points in the intersection of affine subspaces (cf. e.g. \cite{censorzenios}). We refer to \cite{reid,censorzenios} for introductory and exhaustive views on Bregman distances in optimization. 

Although convex functionals play a role in many other branches of mathematics, e.g. in many variational problems and partial differential equations, the  suitability of Bregman distances in such fields was hardly investigated for several descades. In mathematical imaging and inverse problems the situation changed with the rediscovery and further development of Bregman iterations as an iterative image restoration technique in the case of frequently used regularization techniques such as total variation (cf. \cite{osh-bur-gol-xu-yin}), which led to significantly improved results compared to standard variational models and could eliminate systematic errors to a certain extent (cf. \cite{benninggroundstates,gilboa}). Another key observation increasing the interest in Bregman distances in these fields was that they can be employed for error estimation in particular for not strictly convex and nonsmooth functionals (cf. \cite{BuOs04}), which prevent norm estimates. 

Although there are many obvious links to the main route of research in Bregman distances and related optimization algorithms, there are several peculiar aspects that deserve particular discussion. Besides missing smoothness of the considered functionals and the fact that problems in imaging, inverse problems and partial differential equations are naturally formulated in infinite-dimensional Banach spaces such as the space of functions of bounded variation or Sobolev spaces, which have only been considered in few instances before, a key point is that the motivation for using Bregman distances in these fields often differs significantly from those in optimization and statistics. In the following we want to provide an overview of such questions and consequent developments, keeping an eye on potential directions and questions for future research. We start with a section including definitions, examples and some general properties of Bregman distances, before we survey aspects of Bregman distances in inverse problems and imaging developed in the last decade. Then we proceed to a discussion of Bregman distances in partial differential equations, which is less explicit and hence the main goal is to highlight hidden use of Bregman distances and make the idea more directly accessible for future research. Finally we conclude with a section on related recent developments.

\section{Bregman Distances and their Basic Properties}

We start with a definition of a Bregman distance. In the remainder of this paper, let $X$ be a Banach space and 
$J: X \rightarrow \R \cup \{+\infty\}$ be convex functionals. We first recall the definition of subdifferential respectively subgradients.

\begin{definition}
The subdifferential of a convex functional $J$ is defined by
\begin{equation}
	\partial J(u) = \{ p \in X^*~|~J(u) +  \langle p, v-u \rangle \leq J(v) \text{ for all } v \in X \} .
\end{equation}
An element $p \in \partial J(u)$ is called subgradient. 	
\end{definition}

Having defined a subdifferential we can proceed to the definition of Bregman distances, respectively generalized Bregman distances according to \cite{Kiwiel}

\begin{definition}
The (generalized) Bregman distance related to a convex functional $J$ with subgradient $p$ is defined by
\begin{equation}
	D_J^p(v,u) = J(v) - J(u) - \langle p, v- u \rangle,
\end{equation}
where $p \in \partial J(u)$. The symmetric Bregman distance is defined by
\begin{equation}
	D_J^{p,q}(u,v) = D_J^p(v,u) + D_J^q(u,v) = \langle p-q, u-v \rangle,
\end{equation}
where $p \in \partial J(u)$, $q \in \partial J(v)$.
\end{definition}

Note that in the differentiable case, i.e. $\partial J(u)$ being a singleton, we can omit the special subgradient and write $D_J(v,u)$ or $D_J^{J'(u)}(v,u)$.

By the definition of subgradients the nonnegativity is apparent:
\begin{prop}
Let $J$ be convex and $p \in \partial J(u)$. Then 
$$ D_J^p(v,u) \geq 0 \qquad \forall~v \in X $$
and 
$$D_J^p(u,u) = 0. $$ 
If $J$ is strictly convex, then $ D_J^p(v,u) > 0$ for $v \neq u$.
\end{prop}

We can further characterize vanishing Bregman distances as sharing a subgradient:
\begin{prop}
Let $J$ be convex and $p \in \partial J(u)$. Then 
$ D_J^p(v,u) = 0 $
if and only if $p \in \partial J(v)$.
\end{prop}

Since Bregman distances are convex with respect to the first argument, we can also compute a subdifferential with respect to that variable, which is simply a shift of the subdifferential of $J$:
\begin{prop}
Let $J$ be convex, $p \in \partial J(u)$. Then 
$$ \partial_v D_J^p(v,u) = \partial J(v) - p .$$
\end{prop}

Concerning existence proofs for variational problems involving Bregman distance it is often useful to investigate lower semicontinuity properties. Since Bregman distances can be considered as affinely linear perturbations of the functional $J$ it it natural that these properties carry over:
\begin{prop}
Let $J$ be convex and $q \in \partial J(v)$. Then the functional $H$ defined by
$$ H(u) =  D_J^q(u,v)  $$
is convex. Hence, if $X$ is reflexive, then $H$ is weakly lower semicontinuous. If $X$ is the dual of some Banach space $Z$ and $J$ is the convex conjugate of a functional  on $Z$, then $q \in Z$ implies that $H$ is lower semicontinuous in the weak star topology.
\end{prop}

\subsection{Examples of Bregman Distances}

In the following we provide several examples of Bregman distances as frequently found in literature as well as some that received recent attention. This shall provide further insights into the relation to other distance measures and the basic properties of Bregman distances:

\begin{example}Let $X$ be a Hilbert space and $J(u) = \frac{1}2 \Vert u \Vert_X^2$. Then $\partial J(u) = \{u\}$ and hence
\begin{equation}
	D_J^u(v,u) = \frac{1}2 \Vert u - v \Vert_X^2.
\end{equation}
\end{example} 

\begin{example}Let $I$ be a countable index set and $X=\ell^1(I)$ with
$$ J(u)  =  \Vert u \Vert_{\ell^1} = \sum_{i \in I} \vert u_i \vert. $$
Then the Bregman distance is given by
\begin{equation}
	D_J^p(v,u) = \sum_{i \in I} (q_i - p_i) v_i = \sum_{i,v_i > 0} (1-p_i)\vert v_i \vert + \sum_{i,v_i < 0} (1+p_i)\vert v_i \vert.
\end{equation}
Note that the above sums have nonzero entries only if the sign of $u_i$ does not match the sign of $v_i$, since $p_i=1$ if $u_i > 0$ and $p_i=-1$ if $u_i < 0$. 
\end{example}

\begin{example}Let $X=\ell^1_+(\{1,\ldots,N\})$ with
$$ J(u)  =  \sum_{i=1}^N u_i \log u_i + 1 - u_i, $$
which is called the logarithmic entropy (or Boltzmann entropy). 
Then the Bregman distance is given by
\begin{equation}
	D_J^p(v,u) = \sum_{i=1}^N   v_i \log \frac{v_i}{u_i} + u_i - v_i,
\end{equation}
which is known as Kullback-Leibler divergence. 
An analogous treatment applies to $X=L^1_+(\Omega)$, for a bounded domain $\Omega$, and the continuous version 
$$ J(u)  =  \int_\Omega \left( u(x) \log u(x) + 1 - u(x) \right)~dx, $$
resulting in the Bregman distance
\begin{equation}
	D_J^p(v,u) = \int_\Omega \left( v(x) \log \frac{v(x)}{u(x)}+ u(x) - v(x) \right)~dx.
\end{equation}
\end{example}

\subsection{Bregman Distances and Duality}

Duality is a basic ingredient in convex optimization (cf. \cite{ekelandtemam}) and hence it is also interesting to understand some connections 
of duality and Bregman distances. For this sake we employ the convex conjugate (also called Legendre-Fenchel transform) of a functional $J$ given by $J^*: X^* \rightarrow \R \cup \{+\infty\}$ satisfying
\begin{equation}
	J^*(p) = \sup_{u \in X} \left( \langle p, u \rangle - J(u) \right).
\end{equation}

Noticing that for $p \in \partial J(u)$ we have $J^*(p)=\langle p, u \rangle - J(u)$ one can immediately rewrite the Bregman distance as
\begin{equation}
	D_J^p(v,u) = J(v) + J^*(p) - \langle p,v \rangle,
\end{equation}
which can be interpreted as measuring the deviation of $p$ from being a subgradient in $\partial J(v)$ or the deviation of $v$ from being a subgradient in $\partial J^*(p)$.

A key identity relates Bregman distances with respect to $J$ to those with respect to the convex conjugate $J^*$:
\begin{prop}
Let $p \in \partial J(u)$ and $ q \in  \partial J(v)$. Then
\begin{equation}
	D_J^p(v,u) = D_{J^*}^v(p,q). \label{Bregmandistanceduality}
\end{equation}
\end{prop}
\begin{proof}
By simple reordering we find
\begin{eqnarray*}
	D_J^p(v,u) &=& J(v) - \langle p, v\rangle + \langle p, u \rangle - J(u) \\
	&=& J(v) - \langle p, v\rangle + J^*(p),
\end{eqnarray*}
where we have used the maximality relation for the convex conjugate, which is equivalent to $p \in \partial J(u)$. With analogous reasoning we find
$J^*(q) = \langle q,v\rangle - J(v)$ and hence 
\begin{eqnarray*}
	D_J^p(v,u) 	&=& J(v)  + J^*(p) -J^*(q) - \langle p-q, v\rangle = D_{J^*}^v(p,q), 
\end{eqnarray*}	
noticing that $q \in \partial J(v)$ implies $v \in \partial J^*(q)$.
\end{proof}

A second aspect of duality related to Bregman distance is the convex conjugate of the latter, which shows that Bregman distances are dual to measuring differences via a functional:
\begin{prop}
Let $ q \in  \partial J(v)$ and $H$ be defined by
\begin{equation}
 H(u) = 	D_J^q(u,v) .
\end{equation}
Then 
\begin{equation}
	H^*(p) = J^*(p+q)-J^*(q).
\end{equation}
\end{prop}
\begin{proof}
We have
\begin{eqnarray*}
	H^*(p)	&=& \sup_u \left[ \langle p, u \rangle - J(u) + J(v) - \langle q, v - u\rangle \right] \\
	&=& \sup_u \left[ \langle p+q, u \rangle - J(u) \right] - \left[\langle q, v \rangle -J(v) \right] .
\end{eqnarray*}	
The first term equals $J^*(p+q)$ by definition and the second equals $J^*(q)$ since $q \in \partial J(v)$.
\end{proof}

\subsection{Bregman Distances and Fenchel duality} \label{combinationsection} 

In the following we further investigate some properties of Bregman distances for a combination of two convex functionals $F: X \rightarrow \R \cup \{+\infty\}$, $G: Y \rightarrow \R \cup \{+\infty\}$. The  classical setting is related to Fenchel's duality theorem (cf. \cite{ekelandtemam}), where
\begin{equation}
	J(u) := F(u) + G(Ku)  \label{fenchelsetting}
\end{equation}
with $K:X \rightarrow Y$ is a bounded linear operator between Banach spaces. 
The Fenchel duality theorem shows that under suitable conditions 
\begin{equation}
	\inf_u J(u) = \sup_w \left[ F^*(-K^* w) + G^*(w) \right],
\end{equation}
together with equations relating optimal solutions $\hat u$ and $\hat w$ via subdifferentials of the involved functionals
\begin{equation}
	-K^*\hat w \in \partial F(\hat u), \qquad K\hat u \in \partial G^*(w).
\end{equation}
 The above duality opens the possibility to employ Bregman distances on the dual problem as well as on the primal, which is nicely complemented by the duality relations for Bregman distances of a functional and its convex conjugate.

In the following we derive a basic estimates for the variational problem  \eqref{fenchelsetting}, which clarifies the relation of perturbations of one functional with duality and Bregman distances. We shall assume that the regularity of $F$ and $G$ is such that
$$ \partial J(u) = \partial F(u) +  K^* \partial G(Ku) $$
and the Fenchel duality theorem holds (cf. \cite{ekelandtemam} for details). 

Then we obtain the following estimate for perturbations of $J$:
\begin{thm}
Let $F$, $G$ and $K$ be as above, and let $\tilde G$ be a perturbation of $G$ satisfying the same assumptions. Let $u \in X$ be a minimizer of $J$ with $-K^*w \in \partial F(u)$ and  $\tilde u$ be a minimizer of $F(\cdot) + \tilde G(K \cdot  )$ with $-K^*\tilde w \in \partial F(\tilde u)$. Then 
\begin{equation}
	D_F^{-K^*w,-K^*\tilde w}(u,\tilde u) \leq G^*(\tilde w)- G^*(w)+ \tilde G^*(w) - \tilde G^*(\tilde w). \label{GtildeG}
\end{equation}
\end{thm}
\begin{proof}
We have
\begin{eqnarray*}
D_F^{-K^*w,-K^*\tilde w}(u,\tilde u) & =& \langle K^* \tilde w - K^* w, u - \tilde u\rangle \\
&=& \langle Ku, \tilde w - w \rangle + \langle K \tilde u, w - \tilde w  \rangle. 
\end{eqnarray*} 
By the Fenchel duality theorem we have $Ku \in \partial G^*(w)$ and $K\tilde u \in \partial G^*(\tilde w)$, which implies the assertion by inserting the subgradient inequality.
\end{proof}

\subsection{Bregman Distances for One-homogeneous Functionals} \label{onehomogeneoussection}

The case of convex one-homogeneous functionals $J$, i.e.
\begin{equation}
	J(t u) = |t| J(u) \qquad \forall~t \in \R,
\end{equation}
received strong attention recently, and also appears to be a particularly interesting one with respect to Bregman distances. In the one-homogeneous case one has 
\begin{equation}
	J(u) = \langle p, u \rangle 
\end{equation}
for $p \in \partial J(u)$. Thus, the Bregman distance simply reduces to
\begin{equation}
	D_J^p(v,u) = J(v) - \langle p, v \rangle. \label{Bregmanonehomogeneous}
\end{equation}

An interesting property in the one-homogeneous case is the fact that the convex conjugate is the indicator function of a convex set $C$, i.e. ,
\begin{equation}
	J^*(p) =  \left\{ \begin{array}{ll}  0 & \text{if } p \in C \\ + \infty & \text{else}. \end{array} \right.
\end{equation}
This sheds interesting light on \eqref{Bregmandistanceduality}, noticing that $p \in \partial J(u)$ implies $p \in C$. Hence, 
$$ D_J^p(v,u) = D_{J^*}^v(p,q) = \langle q-p,v \rangle. $$
An alternative way to see this property is \eqref{Bregmanonehomogeneous} combined with 
$ \langle q,v \rangle = J(v)$.

In the one-homogeneous case we immediately find an example of Bregman distances vanishing for $v \neq u$. Let $t > 0$ and $v=tu$, then $\partial J(v) = \partial J(u)$ implies
$D_J^p(v,u) = 0$. On the other hand we observe that the Bregman distance distinguishes different orientation. Choosing $v=tu$ for $t < 0$ we have $\partial J(v) = - \partial J(u)$, hence $D_J^p(v,u) = 2 J(v)$.


\section{Applications in Inverse Problems and Imaging}

In the last decade, Bregman distances have become an important tool in inverse problems and image processing. Their main use is twofold: On the one hand they are of particular importance for all kinds of error estimates as already sketched above and in particular they are quite useful for the analysis of variational regularization techniques with nondifferentiable regularization functionals. This route has been initiated in \cite{BuOs04} and subsequently expanded e.g. in \cite{benning,bur07,flemming,grasmair1,grasmair2,hein,poeschl,resmerita,wernerhohage}. On the other hand Bregman distances can be used to construct novel iterative techniques with superior properties compared to classical variational regularization. This route goes back to \cite{osh-bur-gol-xu-yin} and was developed further e.g. in \cite{primaldualbregman,gilboa,cai1,splitbregman,moel12,Xu07,yin,zhang}, the methods also had a huge impact on various applications (cf. e.g. \cite{choi,mueller,primaldualbregman}).

The basic setup we consider is the solution of a problem of the form $K u = f$, where $K:X  \rightarrow Y$ is a bounded linear operator between Banach spaces and $f$ are given data. Since in most cases $K$ does not have a closed range (or is even a compact operator) and data contain measurement errors, this problem can be ill-posed. To cure this issue variational regularization methods employ a convex regularization functional $R:X \rightarrow \R \cup \{+\infty\}$, which introduces the a-priori knowledge that reasonable approximations of the  solution $u$ have small (minimal) values $R(u)$. Variational regularization methods make a compromise between approximating the data $f$ and minimizing $R$ and solve a problem of the form
\begin{equation}
	D(Ku,f) + \alpha R(u) \rightarrow \min_{u \in X}, \label{variationalregularization}
\end{equation}
where $D: Y \times Y \rightarrow \R$ is an appropriate distance measure and $\alpha > 0$ is a regularization parameter to be chosen appropriately in dependence of the measurement error (often refered to as data noise). Specific forms of the distance measure are derived e.g. via statistical modelling as the negative log-likelihood of the data noise. Frequently $D$ is simply a least-squares term, i.e. $Y$ is a Hilbert space and
\begin{equation}
	D(Ku,f) = \frac{1}2 \Vert Ku  -f \Vert_Y^2. 
\end{equation}
A classical example is the ROF-model for image denoising \cite{rof}, where $R$ is the total variation seminorm, $K$ is an embedding from $BV(\Omega)\cap L^2(\Omega)$ into $L^2(\Omega)$, and $D$ the squared $L^2$-norm. For the whole section we shall assume that $D$ is convex with respect to the first variable, which is the case for almost all commonly used examples.

\subsection{Error Estimates} 

Error estimates for solutions of  \eqref{variationalregularization} are of interest with respect to two quantities: First of all, the distance of the data $f$ to the ideal data $Ku^*$, where $u^*$ is the unknown ideal solution. This part is refered to as data error or noise. Secondly, the regularization parameter $\alpha$, which should be equal zero in the case of ideal data and introduces a systematic error in the case of perturbed data (when it needs to be positive).
In the setting of \eqref{fenchelsetting} we thus need to choose
\begin{equation}
	F(u) = \alpha R(u), \qquad G(Ku) = D(Ku,f).
\end{equation}
The optimality conditions for a minimizer $u_\alpha$ are then of the form
\begin{equation}
	   p_\alpha = K^* w_\alpha , \qquad p_\alpha \in \partial R(u_\alpha) \qquad - \alpha K^*w_\alpha 
	\in \partial D(Ku_\alpha,f),
\end{equation}
where the subgradient of $D$ is  meant to be computed with respect to the first argument for fixed $f$.

In order to obtain error estimates for some different data $\tilde f$ we choose $\tilde G(Ku) =  D(Ku,\tilde f)$
and denote by $\tilde u_\alpha$ its corresponding regularized solution with
$$   \tilde p_\alpha = K^* \tilde w_\alpha , \qquad \tilde p_\alpha \in \partial R(\tilde u_\alpha) .$$
Then \eqref{GtildeG} yields
\begin{equation}
		\alpha D_R^{ K^*w_\alpha,K^*\tilde w_\alpha}(u,\tilde u) \leq G^*( \tilde w_\alpha)- G^*( w_\alpha)+ \tilde G^*( w_\alpha) - \tilde G^*( \tilde w_\alpha).
\end{equation}
To further illustrate the behaviour consider the case of a quadratic data fidelity
\begin{equation}
	G(Ku) = D(Ku,f) = \frac{1}2 \Vert Ku - f\Vert^2 \label{Dquadratic},
\end{equation}
for some squared Hilbert space norm, 
which yields $G^*(w) = \frac{1}2 \Vert w \Vert^2 + \langle w, f \rangle$. Hence, 
\begin{equation}
		\alpha D_R^{ K^*w_\alpha,K^*\tilde w_\alpha}(u,\tilde u) \leq \langle f - \tilde f, \tilde w_\alpha - w_\alpha \rangle.
\end{equation}

In the case \eqref{Dquadratic} one can see quite immediately why the (symmetric) Bregman distance is an appropriate error measure for the estimates. Starting with the optimality conditions
\begin{align*}
K u_\alpha - f + \alpha w_\alpha &= 0, &p_\alpha = K^*w_\alpha \in R(u_\alpha), \\
K \tilde u_\alpha - f + \alpha \tilde w_\alpha &= 0, &\tilde p_\alpha = K^* \tilde w_\alpha \in R(\tilde u_\alpha) ,
\end{align*}
we find
\begin{equation}
	 K(u_\alpha - \tilde u_\alpha) + \alpha (w_\alpha-\tilde w_\alpha) = f -f^*. \label{optimalitydifference}
\end{equation}
The right-hand side is exactly the perturbation of the data, whose norm we want to use to estimate errors in the solution $u_\alpha$. Hence we simply take the squared norm on both sides and obtain by expanding on the left-hand side
$$	\Vert K (u_\alpha - \tilde u_\alpha) \Vert^2 + 2 \alpha \langle  w_\alpha - \tilde w_\alpha , K (u_\alpha - \tilde u_\alpha) \rangle + \alpha^2 \Vert w_\alpha - \tilde w_\alpha \Vert^2 = \Vert f - \tilde f \Vert^2. $$
Finally using $K^* w_\alpha = p_\alpha$ we arrive at
\begin{equation}
	\Vert K (u_\alpha - \tilde u_\alpha) \Vert^2 + 2 \alpha D_R^{p_\alpha, \tilde p_\alpha} (u_\alpha, \tilde u_\alpha)   + \alpha^2 \Vert w_\alpha - \tilde w_\alpha \Vert^2 = \Vert f - \tilde f \Vert^2 , \label{Bregmanequality}
	\end{equation}
which implies (by the nonnegativity of all involved terms) the immediate estimate 
\begin{equation}
D_R^{p_\alpha, \tilde p_\alpha} (u_\alpha, \tilde u_\alpha)   \leq \frac{1}{2\alpha} \Vert f - \tilde f \Vert^2
	\end{equation}
for the Bregman distance. Note that \eqref{Bregmanequality} is not just an estimate, but indeed an equality for three error terms - the error in the image of the operator $K$ (somehow the residual), the error in the dual variables $w$, and the Bregman distance of solutions. Here $Ku$ and $w$ are elements of a Hilbert space and it is of course natural to measure their deviations in the corresponding norm, so \eqref{Bregmanequality} yields the Bregman distance as the naturally induced error measure in the Banach space $X$.

Having obtained \eqref{Bregmanequality} it is interesting to note that one can alternatively obtain estimates for two parts of the right-hand side by taking scalar products of \eqref{optimalitydifference} with appropriate elements and subsequent application of the Cauchy-Schwarz respectively Young's inequality. The first is obtained by a scalar product with $Ku_\alpha - Ku^*$, which yields
$$ \Vert K (u_\alpha - \tilde u_\alpha) \Vert^2 +   \alpha D_R^{p_\alpha, \tilde p_\alpha} (u_\alpha, \tilde u_\alpha) = \langle f - \tilde f, K (u_\alpha - \tilde u_\alpha)  \rangle \leq \frac{1}2 \Vert f - \tilde f \Vert^2 + \frac{1}2 \Vert K (u_\alpha - \tilde u_\alpha) \Vert^2,$$
hence 
\begin{equation}
	\Vert K (u_\alpha - \tilde u_\alpha) \Vert^2 + 2 \alpha D_R^{p_\alpha, \tilde p_\alpha} (u_\alpha, \tilde u_\alpha)   \leq \Vert f - \tilde f \Vert^2.
\end{equation}
Using analogous reasoning, a scalar product of \eqref{optimalitydifference} with $w_\alpha - \tilde w_\alpha$ leads to 
\begin{equation}
	 2 \alpha D_R^{p_\alpha, \tilde p_\alpha} (u_\alpha, \tilde u_\alpha)   + \alpha^2 \Vert w_\alpha - \tilde w_\alpha \Vert^2  \leq \Vert f - \tilde f \Vert^2.
\end{equation}

\subsection{Asymptotics}

A key question in inverse problems is the behaviour of the regularized solution $u_\alpha$ as $\alpha \rightarrow 0$, which only makes sense if the noise in the data vanishes, i.e. $f= Ku^*$ for some desired solution $u^* \in X$. It is well-known that for ill-posed problems the convergen.ce can be arbitrarily slow as $\alpha \rightarrow$ without further conditions on the desired solution $u^*$. For a further characterization it is important to note that under appropriate choice of $\alpha$ a limiting solution $u^*$ of the variational model \eqref{variationalregularization} satisfies
\begin{equation}
	R(u) \rightarrow \min_{u \in X} \quad \text{subject to } K u = Ku^*. \label{constrainedproblem}
\end{equation}
This can be seen from the estimate 
$$ D(Ku_\alpha,f) + \alpha R(u_\alpha) \leq D(Ku^*,f) + \alpha R(u^*). $$
Using $\alpha \rightarrow 0$ and $D(Ku^*,f) \rightarrow 0$ we see that $D(Ku_\alpha,f) \rightarrow 0$, 
hence the limit is a solution of $Ku^*=f$. Dividing by $\alpha$ and using nonnegativity of $D$, we find
$$ R(u_\alpha) \leq R(u^*) + \frac{D(Ku^*,f)}\alpha, $$
and under the standard condition on the parameter choice 
$$\frac{D(Ku^*,f)}\alpha \rightarrow 0, $$
we observe that the limit of $u_\alpha$ cannot have a larger value of $R$ than any other solution of $Ku=f$, i.e. it solves \eqref{constrainedproblem}.

 The key observation in \cite{BuOs04,chaventkunisch} is that appropriate conditions in the case of variational regularization is related to the existence of a Lagrange multiplier for \eqref{constrainedproblem}. The Lagrange functional is given by $L(u,w) = R(u) - \langle w, Ku - Ku^* \rangle$, hence the existence of a Lagrange multiplier is the so-called {\em source condition}
\begin{equation}
p^* = K^* w^* \in \partial R(u^*).
	\label{sourcecondition}
\end{equation}
Let us again detail the arguments in the case \eqref{Dquadratic}, where we can indeed use the above error estimates like \eqref{Dquadratic} with $\tilde u_\alpha = u^*$ and $\tilde w_\alpha = w^*$. In order to obtain $u_\alpha$ as the solution of a variational problem we can indeed choose $\tilde f = Ku^* + \alpha w^*$ (note  that \eqref{sourcecondition} is equivalent to the existence of some $\tilde f$ such that $u^*$ solves the variational problem with data $\tilde f$, cf. \cite{BuOs04}). Hence, \eqref{Bregmanequality} becomes
\begin{equation}
	\Vert K (u_\alpha - u^*) \Vert^2 + 2 \alpha D_R^{p_\alpha,p^*} (u_\alpha,u^*)   + \alpha^2 \Vert w_\alpha - w^* \Vert^2 = \Vert f -Ku^* - \alpha w^* \Vert^2 .
	\end{equation}
Again with Young's inequality we end up at
\begin{equation}
	 D_R^{p_\alpha,p^*} (u_\alpha,u^*) \leq \frac{\Vert f -Ku^*   \Vert^2}\alpha + \alpha \Vert w^* \Vert^2,
\end{equation}
which gives the usual optimal choice  $\alpha \sim \Vert f -Ku^*   \Vert$ of regularization parameter in terms of the noise level, exactly as in the linear Hilbert space case (cf. \cite{EnHaNe96}).

\subsection{Bregman Iterations and Inverse Scale Space Methods}

A frequent observation made for variational methods as discussed above is a systematic bias, in particular the methods yield solutions $u_\alpha$ with $R(u_\alpha)$ too small, which e.g. results into a local loss of contrast in the case of total variation regularization (the constrast loss is larger for smaller structures). In order to cure such systematic errors in particular in the case of one-homogeneous regularization it turned out that the well-known Bregman iteration is a perfect tool. Instead of solving the variational problem only once one usually starts at $u_0$ being a minimizer of the regularization functional $R$, i.e. at the coarsest scale (if one agrees that scale is defined by $R$). Then of course $p_0 = 0 \in \partial R(u_0)$ and one can subsequently iterate
\begin{equation}
	u_{k+1} \in \text{arg}\min_{u \in X} \left( D(Ku,f) + \alpha D_R^{p_k}(u,u_k) \right),
\end{equation}
where the subgradient $p_k$ is updated via the optimality condition
\begin{equation}
	p_{k+1} - p_k \in - \frac{1}\alpha K^* \partial D(Ku_k,f).
\end{equation}
Noticing that we can again write $p_k =  K^* w_k$ one can also construct an iteration
\begin{equation}
	w_{k+1} - w_k \in - \frac{1}\alpha \partial D(Ku_k,f),
\end{equation}
from which one can derive the well-kown equivalence to augmented Lagrangian methods for minimizing $R$ subject to $Ku = f$. 

The convergence analysis in the case $f=Ku^*$ follows the well-known route for the Bregman iteration, but due to the ill-posedness of $Ku=f$ there is a particularly interesting aspect in the case of noisy data $f$ differing from the ideal $Ku^*$. If the range of $K$ is not closed, one has to take care of the situation where neither a solution $Ku=f$ nor some kind of least squares solution (a minimizer of $D(Ku,f)$) exists in $X$. Hence, the Bregman iteration has the role of an iterative regularization method and needs to be stopped appropriately before the noise effects start to deteriorate the quality of the solution. Indeed one can show that the Bregman distance $D^{p_k}(u^*,u_k)$ is decreasing during the first iterations up to a certain point when the residual $D(Ku^k,f)$ becomes too small (i.e. one approximates the noisy data stronger than $Ku^*$). Successful stopping criteria as the discrepancy principle are indeed based on comparing the residual with an estimate of the noise $D(Ku^*,f)$ and stop when $D(u^k,f)$ drops below this estimate. 

In imaging a particularly interesting and quite related aspect of Bregman iterations is the scale behaviour. As mentioned above, with scale defined as above by properties of the regularization functional $R$, the Bregman iteration inserts finer and finer scales during its progress. In order not to miss certain scales it is obviously interesting to make small enough steps, which amounts to choosing $\alpha$ sufficiently large. For the limit of $\alpha \rightarrow \infty$ one can interpret the iteration as a backward Euler discretization (with timestep $\frac{1}\alpha$) of a flow, which has been called inverse scale space method by a reminiscence to so-called scale space methods in image processing, which exhibit the opposite scale behaviour (cf. \cite{scalespace,scherzergroetsch}). The inverse scale space flow is a solution of the differential inclusion
\begin{equation}
	\partial_t p(t) \in -K^* \partial D(Ku(t),f)  , \qquad p(t) \in \partial J(u(t)),
\end{equation}
with initial value $u(0)=u_0$ such that $p(0)=0 \in \partial R(u_0)$. It can be interpreted a gradient flow for the subgradient $p$ on a dual functional (cf. \cite{inversetvflow}) or as a doubly nonlinear evolution equation. For the latter we will give an explanation on the analysis in terms of Bregman distances related to the involved functionals in the next section, which is also the appropriate way to analyze the inverse scale space method.

An unexpected result is the behaviour of the inverse scale space flow for polyhedral functions such as the $\ell^1$-norm. Roughly speaking the polyhedral case means that for any $u \in X$ a subdifferential $\partial R(u)$ can be obtained via convex combinations of a finite number of elements (independent of $u$). It has been shown (cf. \cite{adaptive1,adaptive2}) that in such cases and $D(Ku,f) = \frac{1}2 \Vert Ku - f\Vert^2$ the dynamics of the solution $u(t)$ is piecewise constant in time, i.e. quite far from a continuous flow, while the dynamics of the subgradients $p(t)$ is piecewise linear in time. Interestingly, the time steps $t_k$ at which the solution changes can be computed explicitely, and the value of $u(t_k)$ is obtained by minimizing 
$$\Vert Ku - f\Vert^2 \quad \text{ subject to } \text  p(t_k) \in \partial R(u). $$ This is particularly attractive in the case of sparse optimization with $R$ being the $\ell^1$-norm, since the condition $p(t_k) \in \partial R(u)$ defines the sign of $u$ and in particular the set of zeros. This means that the the least-squares problems have to be solved on a rather small support, which is highly attractive for computational purposes (cf. \cite{adaptive1}). Let us briefly explain the behaviour for $R: \R^N \rightarrow \R^+$ being the $\ell^1$-norm and some arbitrary differentiable functional $G$ on the right-hand side, i.e.,
\begin{equation}
	\partial_t p_i(t) = - \partial_{u_i} G(u(t)). 
\end{equation}
In this case the subdifferential is the multivalued sign of $u_i(t))$ and for $u_0=p_0=0$ we obviously find $u_i(t) = 0$ for sufficiently small time since $|p_i(t)|<1$, which holds for all $i$. Hence for $t< t_1$ with $t_1$ to be determined we find
\begin{equation}
	\partial_t p_i(t) = - \partial_{u_i} G(0), 
\end{equation}
which can be integrated easily to
\begin{equation}
	p_i(t_1) = - t_1 \partial_{u_i} G(0). 
\end{equation}
The key observation is that $u_i \neq 0$ for some $i$ is only possible if $|p_i(t_1)| = 1$. This implies that the first time with possibly nonzero $u$ is 
\begin{equation}
	t_1 = \frac{1}{\Vert \partial G(0) \Vert_\infty}.
\end{equation}
At time $t_1$ the sign of all $u_i$ is determined by $p_i(t_1)$ and one can check that a solution is obtained by minimizing
\begin{equation}
	u(t_1) \in \text{arg}\min_{u \in \R^N} G(u) \quad \text{subject to } p_i(t_1) \in \partial |u_i(t_1)|, 
\end{equation}
or in other words
$$ u(t_1) \in \text{arg}\min_{u \in \R^N} G(u) \quad \text{subject to } p_i(t_1) u_i(t_1) \geq |u_i(t_1)|~\forall i. $$ 
The optimality condition for the latter problem can be written as
\begin{equation}
	\partial_{u_i} G(u(t_1)) + \lambda_i( q_i - p_i(t_1)) = 0, \qquad q_i \in \partial |u_i(t_1)|.
\end{equation}
for some $\lambda \in \R^N$ satisfying the complementarity conditions
$$ \lambda_i \geq 0 , \qquad \lambda_i ( p_i(t_1) u_i(t_1) - |u_i(t_1)|)  = 0. $$
This implies $\partial_{u_i} G(u(t_1))  = 0$ of $u_i(t_1) \neq 0$, $\partial_{u_i} G(u(t_1))  \geq  0$ if 
$u_i(t_1) = 0$ and $p_i(t_1) = 1$, and $\partial_{u_i} G(u(t_1))  \leq  0$ if 
$u_i(t_1) = 0$ and $p_i(t_1) = -1$. This implies that we can find a time interval $(t_1,t_2)$ such that 
$$u(t) = u(t_1), \qquad  p(t) = p(t_1) - (t-t_1) \partial G(u(t_1)) $$
is a solution, and $t_2$ is again defined as the minimal time where there exists $i$ such that $|p_i(t_2)|=1$ and $|p_i(t)|<1$. Again, the solution at time $t_2$ is defined by a solution of the variational problem
$$ u(t_2) \in \text{arg}\min_{u \in \R^N} G(u) \quad \text{subject to } p_i(t_2) u_i(t_2) \geq |u_i(t_2)| ~\forall i. $$ 
By an inductive procedure one obtains that the same kind of dynamics goes on for all $t$ until it stops after finite time steps $t_n$ at a minimizer of $G$.

As mentioned above the scale behaviour of the inverse scale space flow is highly attractive in image processing. In the polyhedral case there is a somehow exact decomposition into different scales by the steps made at times $t_k$. Indeed $\partial_t u$ is a sum of concentrated measures in time, and one may eliminate certain scales by leaving out the corresponding jump $u(t_k+\tau)-u(t_k-\tau).$ This observation leads the way to a much a more general definition of filters from the inverse scale space method, which was discussed in \cite{spectraltv}
\begin{equation}
	\partial_t p(t) = f - u(t), \qquad p(t) \in \partial R(u(t)).
\end{equation}
A certain scale filter is defined by
\begin{equation}
	F(f) = u_0 + \int_0^\infty w(t) d\partial_t u(t),
\end{equation}
with measureable weights $w(t) \in [0,1]$. In the case $w\equiv 1$ one simply obtains $f$, while certain scales can be damped out choosing $w(t) = 0$ for $t$ in an appropriate interval. The design of filters for certain purpose is an ongoing subject of research. 


\section{Applications in Partial Differential Equations}  

In the following we provide an overview of different aspects of partial differential equations, where Bregman distances are a useful tool. Unlike the case of inverse problems and image processing discussed above the notion of Bregman distance is not used widely in this field, and indeed most applications do not refer to this term or use it in a very hidden way. Our goal in the following section is to work out the basic ideas related to Bregman distances in a structured way, which sheds new light on many established techniques and hopefully also opens routes towards novel results. For this sake we employ a formal approach and avoid technicalities such as detailed function spaces, which of course can be worked out from existing literature.

\subsection{Entropy Dissipation Methods for Gradient Systems} 

Entropy dissipation methods are a frequently used tool in partial differential equations (cf. \cite{essay,juengelentropie}), which is often based on using the logarithmic entropy
\begin{equation}
	E(u) = \int_\Omega u(x) \log u(x) ~dx
\end{equation}
as a Lyapunov functional, e.g. in diffusion equations (cf. e.g. \cite{arnoldcarlen,essay,arnold,arnoldunterreiter,carrillo}), kinetic equations (cf. e.g. \cite{essay}), or fluid mechanics (cf. e.g. \cite{saintraymond}). In particular in gradient systems also different convex functionals are used regularly and in a structured way. The abstract form of a gradient system is
\begin{equation}
	\partial_t u(t) = - L(u(t)) E'(u(t)), 
\end{equation}
where $L(u)$ is a linear symmetric positive semi-definite operator on appropriate spaces and $E$ a convex energy functional, which we assume differentiable for simplicity (similar treatment for non-differentiable convex functionals is possible by using subgradients, but beyond our scope). The entropy dissipation property can be verified by the straight-forward computation
\begin{equation}
	\frac{d}{dt} E(u(t)) = E'(u(t)) \partial_t u(t) = - \langle E'(u(t)), L(u(t)) E'(u(t)) \rangle \leq 0. 
\end{equation}
The negative of the right-hand side is frequently called entropy dissipation functional $D(u(t))$ and can be used to derive further quantitative information about the decay to equilibrium. 
A standard example (cf. \cite{essay,carrillo}) are nonlinear Fokker-Planck equations of the form
\begin{equation}
	\partial_t u = \nabla \cdot ( m(u) \nabla (e'(u) + V)) 
\end{equation}
on a domain $\Omega \subset \mathbb{R}^d$ with no-flux boundary conditions. Here, $e:\mathbb{R}^+ \rightarrow \mathbb{R}$ is a convex function, $m:\mathbb{R}^+ \rightarrow \mathbb{R}^+$ a (potentially nonlinear) mobility function, and $V: \Omega \rightarrow \mathbb{R}$ an external potential. Recently also systems of Fokker-Planck equations as well as certain reaction-diffusion systems of the form
\begin{equation}
	\partial_t u_i = D_i \Delta u_i + F_i(u_1,\ldots,u_M), \qquad i=1,\ldots,M
\end{equation}
have been investigated with entropy dissipation techniques (cf. \cite{stelzer,markowichfellner,mielke}).

The major purpose of entropy dissipation techniques is to obtain qualitative or ideally quantitative results about the decay to equilibrium for transient solutions. An equilibrium solution $u_\infty$ is a minimizer of $E$ on a convex set $K$, to which also the transient solution $u(t)$ belongs for all $t$. An example is the Fokker-Planck equation with linear mobility $m(u)=u$, where $K$ is the set of nonnegative integrable functions with prescribed mean vaule.
Hence, $u_\infty$ satisfies
\begin{equation}
	E'(u_\infty) (u-u_\infty) \geq 0 \qquad \forall~u \in K.
\end{equation}
If further the operator $L(u)$ is such that 
\begin{equation}
	L(u) E'(u_\infty) = 0 \qquad \forall~u \in K,
\end{equation}
which is indeed the case for the typical examples, 
then one can rewrite the gradient system as
\begin{equation}
	\partial_t u(t) = - L(u(t)) (E'(u(t))-E'(u_\infty)).
\end{equation}
Hence, the right-hand side is expressed in a difference of energy gradients for the transient and equilibrium solution.
In a similar way, the entropy dissipation can be rewritten in terms of a distance between those and the Bregman distance (usually called relative entropy) plays a key role for this purpose. One observes that 
\begin{eqnarray}
	\frac{d}{dt} D_E^{E'(u_\infty)}(u(t),u_\infty) &=& E'(u(t)) \partial_t u(t) - E'(u_\infty) \partial_t u(t) \nonumber \\ &=& - \langle E'(u(t)) - E'(u_\infty) , L(u(t)) ( E'(u(t)) - - E'(u_\infty) )\rangle \nonumber \\ &=:& - F(u(t),u_\infty). \nonumber
\end{eqnarray}
Of course, the above computation holds for smooth solutions only, for weak solutions on can usually derive the time-integrated version
\begin{equation}
	D_E^{E'(u_\infty)}(u(t),u_\infty) + \int_s^t F(u(\tau)~d\tau \leq D_E^{E'(u_\infty)}(u(s),u_\infty).
\end{equation}

The above computation shows that entropy dissipation can be rephrased as the decrease of the Bregman distance between stationary and transient solution. We notice that the use of the Bregman distance is not essential in this computation, but the understanding of this structure can be quite benefitial, in particular if one wants to use dual variables, the so-called entropy variables
\begin{equation}
	\varphi(t) = E'(u(t)), \qquad \varphi_\infty = E'(u_\infty) .
\end{equation}
The entropy variable $\varphi$ solves the system
\begin{equation}
	\partial_t (E'(\varphi(t))= - L((E^*)'(\varphi(t))) \varphi(t) , \label{dualgradientflow}
\end{equation}
where $E^*$ is the convex conjugate of $E$. When analyzing the dual flow \eqref{dualgradientflow} a dissipation property can now be derived immediately using relation \eqref{Bregmandistanceduality}. Thus, we obtain a dual entropy dissipation of the form
\begin{equation}
	\frac{d}{dt} D_{E^*}^{u(t)}(\varphi_\infty,\varphi(t)) = - - \langle \varphi(t) - \varphi_\infty ,  L((E^*)'(\varphi(t))) ( \varphi(t) - \varphi_\infty )\rangle.
\end{equation}
The duality relation is particularly interesting for constructing approximations in terms of the entropy variables, as e.g. carried out for degenerate cross-diffusion systems in (cf. \cite{schlake,juengelbounded,zamponi}).

In order to obtain quantitative estimates for the decay one needs

\subsection{Lyapunov Functionals for Gradient Systems out of Equilibrium}

The appropriate use of Bregman distances seems to be less explored, but maybe even more crucial for the derivation of Lyapunov functionals if gradient systems are perturbed out of equilibrium. The simplest example is the linear Fokker-Planck equation with non-potential force as investigated in \cite{arnoldcarlen}
\begin{equation}
	\partial_t u = \nabla \cdot ( \nabla u + u F) \qquad \text{in } \Omega \times \mathbb{R}^+, \label{Fokkerplancknonsymmetric}
\end{equation}
supplemented by no-flux boundary conditions
\begin{equation}
	(\nabla u + u F)\cdot n = 0 \qquad \text{on } \partial \Omega \times \mathbb{R}^+.
\end{equation}
If the vector field $F$ is not the gradient of some potential function, then a stationary solution cannot be constructed as the minimizer of an entropy functional. However, the existence and uniqueness of a stationary solution can be shown under quite general assumptions on $F$ (cf. \cite{droniou}). In a form similar to gradient flows we write \eqref{Fokkerplancknonsymmetric} as
\begin{equation}
	\partial_t u = \nabla \cdot ( u (\nabla e'(u) + F)), \qquad e(u) = u \log u +1 - u, 
\end{equation}
which suggests to further investigate distances based on the entropy functional
\begin{equation}
	E(u) = \int_\Omega e(u)~dx = \int_\Omega (u \log u - u  +1)~dx.
\end{equation}
The dissipation of the relative entropy can be computed via
\begin{eqnarray*}
 \frac{d}{dt} D_E^{E'(u_\infty)}(u(t),u_\infty) &=& \int_\Omega (e'(u(t))-e'(u_\infty))\partial_t u(t) ~dx \\
&=& \int_\Omega (e'(u(t))-e'(u_\infty))\nabla \cdot u (\nabla (e'(u(t)) - e'(u_\infty)) + \nabla e'(u_\infty) + F) ~dx \\
&=& - \int_\Omega u|\nabla (e'(u(t)) - e'(u_\infty))|^2~dx \\ &&
+ \int_\Omega (e'(u(t))-e'(u_\infty))\nabla \cdot u(t) (\nabla e'(u_\infty) + F)~dx,
\end{eqnarray*}
where we have used the no-flux boundary conditions 
$$ (\nabla e(u(t)) + F) \cdot n = (\nabla e(u_\infty) + F) \cdot n = 0 \qquad \text{on } \partial \Omega \times \mathbb{R}^+
$$
in order to apply integration by parts in the first term on the right-hand side. The second term is simplified via
\begin{eqnarray*}
 \nabla \cdot (u (\nabla e'(u_\infty) + F)) &=& \nabla \cdot (\frac{u(t)}{u_\infty} u_\infty (\nabla e'(u_\infty) + F)) \\
&=& u_\infty  \nabla  (\frac{u(t)}{u_\infty}) \cdot (\nabla e'(u_\infty) + F)) \\
&=& u_\infty  \nabla  \exp(e'(u(t))-e'(u_\infty)) \cdot (\nabla e'(u_\infty) + F)) \\
&=& u_\infty  \exp(e'(u(t))-e'(u_\infty)) \nabla (e'(u(t))-e'(u_\infty)) \cdot (\nabla e'(u_\infty) + F)) .
 \end{eqnarray*}
With $\Psi$ satisfying $\Psi'(z) = z \exp(z)$ we can further write
\begin{eqnarray*}
&&\int_\Omega (e'(u(t))-e'(u_\infty))\nabla \cdot u(t) (\nabla e'(u_\infty) + F)~dx = \\
&& \qquad  \int_\Omega \nabla \Psi(e'(u(t))-e'(u_\infty)) \cdot u_\infty (\nabla e'(u_\infty) + F)~dx = 0, 
 \end{eqnarray*}
which can be seen again through integration by parts. Hence, we finally obtain the decrease of the Bregman distance via
\begin{equation}
	\frac{d}{dt} D_E^{E'(u_\infty)}(u(t),u_\infty) = -  \int_\Omega u|\nabla (e'(u(t)) - e'(u_\infty))|^2~dx ,
\end{equation}
and the logarithmic Sobolev inequality (cf. \cite{arnold}) implies exponential convergence to equilibrium.

Another example are boundary-driven nonlinear Fokker-Planck equation
\begin{equation}
	\partial_t u = \nabla \cdot ( \nabla m(u) + m(u) F) \qquad \text{in } \Omega \times \mathbb{R}^+, \label{eqbodineau}
\end{equation}
considered in \cite{bodineau} with Dirichlet boundary conditions
\begin{equation}
	u =g \qquad \text{on } \partial \Omega \times \mathbb{R}^+.
\end{equation}
We mention that an analogous analysis holds in the case of no-flux boundary conditions (in which case we have a direct generalization of the nonsymmetric Fokker-Planck equation above) or mixed Dirichlet and no-flux boundary conditions. 
Bodineau et al \cite{bodineau} construct Lyapunov functionals of the form
\begin{equation}
	H(u,u_\infty) = \int_\Omega \int_{u_\infty(x)}^{u(x,t)} \Phi'\left(\frac{m(s)}{m(u_\infty(x))}\right)~ds~dx,
\end{equation}
where $\Phi$ is a nonnegative function with unique minimum at zero. Such a construction seems far from being intuitive, but it becomes much more clear for $\Phi$ being the logarithmic entropy, i.e. $\Phi'(t) = \log t$. In this case the Lyapunov functional becomes
\begin{equation}
	H(u,u_\infty) = \int_\Omega \int_{u_\infty(x)}^{u(x,t)}\log m(s)- \log m(u_\infty(x))~ds~dx,
\end{equation}
and with a function $e$ such that $e'(s) = \log m(s)$ we further obtain
\begin{equation}
	H(u,u_\infty) = \int_\Omega (e(u(x,t)) - e(u_\infty(x))- e'(u_\infty(x)) (u(x,t) - u_\infty(x)))~ds~dx,
\end{equation}
which is nothing but the Bregman distance for the entropy functional
\begin{equation}
	E(u) = \int_\Omega e(u)~dx, \qquad \text{with } e'(u) = \log m(u).
\end{equation}
Since equation \eqref{eqbodineau} can be written as 
\begin{equation}
	\partial_t u = \nabla \cdot ( m(u) (\nabla \log m(u) +  F)), \qquad \text{in } \Omega \times \mathbb{R}^+,
\end{equation}
the above form of $E$ is also a natural choice. The detailed computations for the entropy dissipation are indeed completely analogous to the case of the linear Fokker-Planck equation, the crucial point appears to be the logarithmic relation between entropy derivatives $e'(u)$ and mobilities $m(u)$.

\subsection{Doubly Nonlinear Evolution Equations}

A generalization of gradient systems are doubly nonlinear evolution equations with a gradient structure either of the form 
\begin{equation}
	\partial_t p(t)  \in - \partial G(u(t)), \qquad p(t) \in \partial F(u(t))  \label{eq:gradstructure}
\end{equation}
or as 
\begin{equation}
	\partial F(\partial_t u) + \partial G(u(t)) \ni 0 \label{eq:gradstructure2}.
\end{equation}
The best studied case, which is also the one where both coincide, corresponds to $F(u) = \frac{1}2 \Vert u \Vert^2$ for a norm in a Hilbert space, which yields the classical gradient flow 
\begin{equation}
	\partial_t u(t)  \in - \partial G(u(t)). \label{eq:gradflow}
\end{equation}
We have seen a system in the form \eqref{eq:gradstructure} already above in the inverse scale space method, while the form \eqref{eq:gradstructure2} appears frequently in mechanical problems (cf. e.g. \cite{mielkerossi} and references therein)
There is indeed a duality relation  for \eqref{eq:gradstructure} and 
\eqref{eq:gradstructure2}. Starting from \eqref{eq:gradstructure} we obtain 
$u(t) \in \partial G^*(-\partial_t p(t)) \cap \partial F^*(p(t))$, respectively $- u(t) \in \partial G^*(\partial_t p(t)) $ if $G$ satisfies a symmetry condition around zero. This yields 
$$\partial G^*(\partial_t p(t)) + \partial F^*(p(t)) \ne 0, $$
the analogue of \eqref{eq:gradstructure2}.

Doubly nonlinear evolution equations have recently been investigated extensively, and in particular tools from convex analysis have been employed (cf. \cite{mielkerossi}). Here we add our Bregman distance point of view to derive estimates for such equations.
Let us start with a straightforward computation on the change of the time derivative of the Bregman distance:
\begin{lemma}
Let $F$ be differentiable and $u$   a solution of \eqref{eq:gradstructure}. Then 
$$ \frac{d}{dt} D_F^{p(t)}(v,u(t)) = - \langle \partial_t p(t), v -u(t)\rangle
 \leq G(v) - G(u(t)). $$
\end{lemma}

This can be used to quantify the distance of $u(t)$ to a minimizer of $G$:
\begin{cor}
Let $F$ be differentiable, $u_\infty$ a minimizer of $G$, and $u$ a solution of \eqref{eq:gradstructure}. Then 
\begin{equation}
	\frac{d}{dt} D_F^{p(t)}(u_\infty,u(t)) + D_G^0(u(t),u_\infty) \leq 0 . \label{dFestimate}
\end{equation}
\end{cor} 

Since it is straightforward to see
\begin{equation}
	\frac{d}{dt} D_G^0(u(t),u_\infty) = \frac{d}{dt} G(u(t)) \leq 
\end{equation}
we see after integrating \eqref{dFestimate} in time
\begin{equation}
	D_F^{p(t)}(u_\infty,u(t)) + t  D_G^0(u(t),u_\infty)  \leq D_F^{p(t)}(u_\infty,u(t)) + \int_0^t  D_G^0(u(s),u_\infty) ~ds \leq D_F^{p(0)}(u_\infty,u(0)), 
\end{equation}
leading to linear decay of the Bregman distance:

\begin{thm}
Let $F$ be differentiable, $u_\infty$ a minimizer of $G$, and $u$ a solution of \eqref{eq:gradstructure}. Then 
\begin{equation}
	D_G^0(u(t),u_\infty) \leq \frac{1}t D_F^{p(0)}(u_\infty,u(0)) . \ 
\end{equation}
\end{thm}

\subsection{Error Estimates for Nonlinear Elliptic Problems}

We finally turn our attention to the analysis of discretization methods for nonlinear elliptic problems such as the $p$-Laplace equation. Such elliptic problems are optimality conditions of some energy functional of the form
\begin{equation}
	E(u) = J(u) - \langle f, u \rangle,
\end{equation}
where $J$ is a convex functional on a Banach space $X$, typically a Sobolev space of first order derivatives. The elliptic differential equation (or more general differential inclusion) is the optimality condition
\begin{equation}
	p = f, \qquad p \in \partial J(u)
\end{equation}
A canonical example is the $p$-Laplace equation
\begin{equation}
	-\nabla \cdot \left( |\nabla u|^{p-2} \nabla u \right) = f,
\end{equation}
which is related to the functional
\begin{equation}
 J(u) =	\frac{1}p \int_\Omega |\nabla u(x)|^p~dx .
\end{equation}

For variational discretizations of such problems the Bregman distance appears to be a quite useful tool, which is still not fully exploited. In many approaches the Bregman distance is used in a hidden way and strict convexity is used to obtain an estimate in terms of the underlying norms (with potentially suboptimal constants however). For the $p$-Laplace equation such an approach is carried out in \cite{dieningkreuzer}. Again in the limiting case $p=1$ related to total variation minimization the Bregman distance is even more crucial and appears e.g. in \cite{bartels}. Here we briefly sketch the obvious role of Bregman distances in Galerkin discretizations of the form
\begin{equation}
	E(u) \rightarrow \min_{u \in X_h}, 
\end{equation}
where $X_h$ is a finite-dimensional subspace of $X$, e.g. constructed by finite elements. 

Let us start by pointing out the basic structure of error estimates for Galerkin methods in the linear, case related to the minimization of a positive definite quadratic form
\begin{equation}
	J(u) = B(u,u),
\end{equation}
where $B: X \times X \rightarrow \mathbb{R}$ is a bounded and coercive bilinear form. The optimality condition in weak form is given by
\begin{equation}
	B(u,v) = \langle f, v \rangle \qquad \forall~v \in X,
\end{equation}
and the Galerkin discretization yields a solution $u_h \in X_h$ of 
\begin{equation}
	B(u_h,v) = \langle f, v \rangle \qquad \forall~v \in X_h.
\end{equation}
Error estimates for such discretizations are obtained in two steps: first the error between $u$ and $u_h$ is estimated by the projection error to the subspace $X_h$ and then the projection error is estimated, e.g. via the interpolation error. The crucial property for the first step is the so-called {\em Galerkin orthogonality} 
\begin{equation}
	B(u-u_h, v ) = 0 \qquad \forall~v \in X_h,
\end{equation}
which implies
\begin{equation}
	B(u-u_h, u-u_h ) = B(u-u_h, u-v ) \qquad \forall~v \in X_h,
\end{equation}
and by the Cauchy-Schwarz inequality for the positive definite bilinear form $B$
\begin{equation}
	B(u-u_h, u-u_h ) \leq B(u-v, u-v ) \qquad \forall~v \in X_h.
\end{equation}
In other words $u_h$ is the projection of $u$ on the subspace $X_h$, when the (squared) norm induced by $B$ is used as a distance measure.

Since the term $B(u-v,u-v)$ above is just the Bregman distance related to quadratic functional $J$ one might think of an analogous property in the case of nonquadratic $J$, when the Bregman projection is used. Indeed, we can derive such a relation in the case of arbitrary convex $J$. For this sake let again $u$ be a minimizer of $E$ and $u_h$ a minimizer of $E$ constrained to the subspace $X_h$. Then we have $f \in \partial J(u)$ and thus, since $u_h$ minimizes $E$ on $X_h$, we have for all $v \in X_h$
\begin{eqnarray*}
D_J^f(u_h,u) &=& J(u_h) - J(u) - \langle f, u_h - u \rangle \\ &=& E(u_h) - J(u) + \langle f, u\rangle \\
&\leq& E(v) - J(u) + \langle f, u\rangle.
\end{eqnarray*}
Rewriting the last term we hence obtain the Bregman projection property
\begin{equation}
	D_J^f(u_h,u) \leq D_J^f(v,u), \qquad \forall~v \in X_h.
\end{equation}
This observation opens a way to analyze Galerkin methods for such nonlinear problems in the same way as in the linear case, the key step to be developed for specific problems and specific discretizations ($X_h$) is the estimation of the Bregman projection error. 

Note again the role of the Bregman distance for error estimation: The one-sided distance $D_J^f(u_h,u)$ is particularly suitable for the estimation of a-priori errors as above, while a-posteriori error estimation should rather be based on the distance $D_J^{p_h}(u,u_h)$ with $p_h \in \partial J(u_h)$. We have by the minimizing property of $u$
\begin{eqnarray*}
D_J^{p_h}(u,u_h) &=& J(u) - J(u_h) - \langle p_h, u-u_h \rangle \\ &=& E(u) - E(u_h) + \langle p_h-f, u_h - u \rangle \\
&\leq& \langle p_h-f, u_h - u \rangle .
\end{eqnarray*}
Using the duality relation $u \in \partial J^*(f)$, this could be further estimated to the full a-posteriori estimate
\begin{equation}
	D_J^{p_h}(u,u_h) \leq \langle p_h-f, u_h \rangle + J^*(2 f -p_h)- J^*(f) .
\end{equation}
For practical purposes the above abstract estimate is not useful in most cases, since computing the adjoint $J^*$ means to solve a nonlinear partial differential equation as well, which might be as difficult as the original one. However, the general strategy can be exploited together with specific properties of the functional $J$ and the subspace $X_h$. In particular for gradient energies of the form
\begin{equation}
	J(u) = \int_\Omega j(\nabla u)~dx
\end{equation}
one can derive alternative versions using only the convex conjugate $j^*$, which is significantly easier to compute.

\section{Further Developments}

In this final section we discuss some aspects of Bregman distances that came up recently and will potentially have strong further impact, in particular we will explore some developments related to probability.

\subsection{Uncertainty Quantification in Inverse Problems}

Since Bregman distances appear to be a suitable tool for estimates in certain nonlinear deterministic problems, it seems natural to exploit them also in the  stochastic counterparts of such problems. The obvious measure for error estimates is then the expected value of the Bregman distance with respect to the stochastic quantity. Such approaches have been used successfully in particular in statistical inverse problems (cf. e.g. \cite{wernerhohage}), which we also want to discuss in the following. In order to avoid technicalities we restrict ourselves to a purely finite-dimensional setup.

Consider the inverse problem $Ku = f$, where $K: \mathbb{R}^N \rightarrow \mathbb{R}^M$ and the data are generated from a true solution $u^*$ with additive Gaussian noise, i.e.
\begin{equation}
	f = Ku^* + \sigma n,
\end{equation}
with $n$ a Gaussian random variable with zero mean and covariance matrix $I_M$. Let again $R$ be a convex regularization functional and $u_\alpha$ a solution of the variational problem
\begin{equation}
 J(u) =	\frac{1}{2 \sigma^2} \Vert (Ku-f)\Vert^2 +  \alpha J(u) \rightarrow \min_{u \in \mathbb{R}^N}. 
\end{equation}
Then $u_\alpha$ satisfies the optimality condition
\begin{equation}
	\frac{1}{ \sigma^2} K^*  K (u_\alpha - u^*) + \alpha p_\alpha = \frac{1}{ \sigma^2} K^* n, \qquad p_\alpha \in \partial R(u_\alpha), \label{uqequation1}
\end{equation}
which implies $p_\alpha = K^*w_\alpha$. Now assume $u^*$ satisfies the source condition \eqref{sourcecondition} then we have
$$ 
  K (u_\alpha - u^*) + \alpha \sigma^2  ( w_\alpha -w ^*) =  n - \alpha \sigma^2 w^*.
$$
Taking the squared norm and subsequently expection with respect to $w$ in this identity we obtain
\begin{eqnarray*} 
 2\alpha \sigma^2 E[D_R^{p_\alpha,p^*}(u_\alpha,u^*)] &\leq&
  E[\Vert K (u_\alpha - u^*) \Vert^2 + 2\alpha \sigma^2 D_R^{p_\alpha,p^*}(u_\alpha,u^*) + \alpha^2 \sigma^4 \Vert w_\alpha -w ^* \Vert]^2  \\ &=&  E[\Vert n - \alpha \sigma^2 w^* \Vert^2] \\
	&=& E[\Vert n \Vert^2] +   \alpha^2 \sigma^4 \Vert w^* \Vert^2 = \sigma^2 M +  \alpha^2 \sigma^4 \Vert w^* \Vert^2  
\end{eqnarray*}
Thus, the expected error in the Bregman distance is estimated by
\begin{equation}
	 E[D_R^{p_\alpha,p^*}(u_\alpha,u^*)]  \leq \frac{M}{2\alpha} + \frac{\alpha \sigma^2}2 \Vert w^* \Vert^2  . 
\end{equation}
We notice that the above approach not only yields an estimate of the Bregman distance, but indeed an exact value for the sum of three error measures, in addition to the Bregman distance also the residual error as well as the error in the source space (related to $w_\alpha - w^*$). Usually the latter is the largest of the three, so one needs to expect a blow up of this term as $M\rightarrow \infty$ if $\alpha$ is not increasing as $M$. If one is interested in the first two terms only, one can simply use a duality product with $u_\alpha - u^*$ in \eqref{uqequation1} and subsequently estimate the expected value of the right-hand side in a different way, which may lead to robust estimates in terms of $M$ respectively estimates that can be carried out for infinite-dimensional white noise. 

An application of Bregman distances in Bayesian modelling was recently investigated in \cite{BL14}, considering frequently used posterior densities of the form
\begin{equation}
	\pi(u|f) \sim e^{-\frac{\Vert Ku-f\Vert^2}{2\sigma^2} - \alpha R(u)}, 
\end{equation}
where again $R$ is a convex and Lipschitz continuous functional on $\R^N$ (generalizations to posterior distributions in infinite-dimensional spaces where further studied in \cite{helin}). It has been shown that the posterior can be centered around the so-called maximum a-posteriori probability (MAP) estimate $\hat u$, which maximizes $p(u|f)$, in the form
\begin{equation}
	\pi(u|f) \sim e^{-\frac{\Vert Ku-K\hat u\Vert^2}{2\sigma^2} - \alpha D_R^{\hat p}(u,\hat u)}.
\end{equation}
Based on the observation 
\begin{equation}
	 \langle s, u - \hat u \rangle = \frac{\Vert Ku-K\hat u\Vert^2}{\sigma^2} + \alpha \langle p -\hat p, u- \hat u \rangle.
\end{equation}
for $p\in \partial R(u)$ and
$$s = \frac{1}{\sigma^2}K^*(Ku-f) + \alpha p \in \partial ( - \log \pi(u|f)), $$
a Bayes cost of the form
\begin{equation}
\Gamma(v) = 	{\bf E}_{p(u|f)} \left[ \frac{\Vert Ku-Kv\Vert^2}{\sigma^2} + \alpha \langle q - \hat p, v- \hat u \rangle \right]
\end{equation}
has been introduced for $ q \in \partial R(v)$ (note that selection of $p \in R(u)$ is only needed on a set of zero measure due to Rademacher's theorem). A simple integration by parts argument then shows that the MAP-estimate $\hat u$ is a minimizer of the Bayes cost, which is a quite natural choice compared to the highly degenerate cost usually used to characterize MAP estimates (cf. \cite{kaipio}). A direct consequence is the fact that the MAP estimate has smaller Bregman distance in expectation than the frequently used conditional mean estimate, hence one obtains a theoretical argument explaining the success of MAP estimates in practice.

\subsection{Bregman Distances and Optimal Transport}

Bregman distances can be used also as a cost in optimal transport, which has been investigated in \cite{CarlierJimenez} for a convex and differentiable functional $J$ on $\R^N$. Given two probability measures $\mu$ and $\nu$, an optimal transport plan is a probability measure $\gamma$ on $\R^N \times \R^N$ with marginals $\mu$ and $\nu$ minimizing the functional
\begin{equation}
F(\gamma) =	\int_{\R^N \times R^N} D_J^{J'(u)}(v,u) ~d\gamma(v,u).
\end{equation}
The resulting optimal value of $F$ can be interpreted as a transport distance between the measures $\mu$ and $\nu$. 

Besides the important question of well-posedness solved in (cf. \cite{CarlierJimenez}) there are several interesting problems such as the existence of transport maps under certain condition (i.e. concentration of $\gamma$ on a set described by the graph of a map $T: \R^N \rightarrow \R^N$) as well as relations to uncertainty quantification. A first example is the Bayes cost approach described in the previous section, which can indeed be interpreted as the transport distance between the posterior distribution and a measure concentrated at the MAP estimate. This motivates further research in the future, an obvious next step might be to estimate distances between different posterior distributions in transport distances related to Bregman distances. 

A different use of Bregman distances in optimal transport was recently made in \cite{BenamouCarlier} for the solution of Monge-Kantorovich formulations in optimal transport. They consider entropic regularizations of the problem, i.e. for $\epsilon > 0$ they minimize a discrete version of
\begin{equation}
F_\epsilon(\gamma) =	\int_{\R^N \times R^N} C(v,u) ~d\gamma(v,u) + \epsilon E(\gamma),
\end{equation}
where $E$ is the entropy
\begin{equation}
	E(\gamma) = \int_{\R^N \times R^N} \log\left(\frac{d\gamma}{d{\cal L}}\right) ~d\gamma(v,u), 
\end{equation}
where $\frac{d\gamma}{d{\cal L}}$ is the Radon-Nikodym derivative with respect to the Lebesgue measure. The key observation is that the minimization of $F_\epsilon$ can be rewritten equivalently as the minimization of the Kullback-Leibler divergence, i.e. the Bregman distance related to $E$, between $\gamma$ and the Gibbs measure $\varphi_\epsilon$ with density $e^{-C/\epsilon}$
\begin{equation}
	D_E(\gamma,\varphi_\epsilon) \rightarrow \min_\gamma,
\end{equation}
which transforms the problem into a Bregman projection problem of the Gibbs density onto the set of plans with given marginals, which can be computed much more efficiently than the original transport control problem. Note that the general procedure can be carried out as well with an arbitrary convex functional whose domain are positive densities, the corresponding Gibbs density is then to be defined as
$\varphi_\epsilon =(E^*)'(-C/\epsilon)$. A particular computational advantage of the logarithmic entropy is the fact that iterative Bregman projections can be computed explicitely and realized with low complexity, in the discrete sets it only needs multiplications and scalar products of diagonal matrices with the matrix discretizing the Gibbs measure (cf. \cite{BenamouCarlier} for further details).

\subsection{Infimal Convolution of Bregman Distances}

Infimal convolution of convex functionals become popular recently in image processing in order to combine favourable properties of certain regularization functionals, e.g. total variation and higher-order versions thereof (REFs). A quite unexplored topic is the infimal convolution of Bregman distances however. Since they are convex functionals of the first variable one may consider the infimal convolution
\begin{equation}
	[D_R^{p_1}(\cdot,u_1) \square D_R^{p_2}(\cdot,u_2)](u) = \inf_{v \in X} [D_R^{p_1}(u-v,u_1) + D_R^{p_2}(v,u_2)],
\end{equation}
with an obvious extension to more than two values.

Of particular interest in imaging applications appears to be the case of $p_2=-p_1$ and $u_2 = - u_1$ for a one-homogeneous functional such as total variation. The latter was used to obtain a regularization functional enforcing partly equal edge sets (REF colorbregman). While minimizing the Bregman distance for total variation strongly favours edge sets with jumps of equal sign (see also the discussion related to orientation for one-homogeneous functionals in Section \ref{onehomogeneoussection}), the infimal convolution of Bregman distances eliminates this part and hence measures differences in edge sets rather than jumps of the same sign. A further study of theoretical properties as well as applications of such kind of infimal convolution of Bregman distances remains an interesting property for future research. One obvious candidate are problems in compressed sensing where one is first of all aims at obtaining the correct support of the solution rather than the sign.

\section*{Acknowledgements}

This work was partially supported by ERC via Grant EU FP 7 - ERC Consolidator
Grant 615216 LifeInverse and by the German Science Foundation DFG
via BU 2327/6-1 and EXC 1003 Cells in Motion Cluster of Excellence, M\"unster,
Germany


\begin{thebibliography}{}

\bibitem{AmbrosioGigliSavare}
{\sc L.~Ambrosio, N.~Gigli, G.~Savare}, {\em Gradient Flows in Metric Spaces and in the
Space of Probability Measures} Birkhaeuser, Basel, Boston, Berlin, 2005.

\bibitem{arnoldcarlen}
{\sc A.~Arnold, E.~Carlen, Q.Ju }, {\em Large-time behavior of non-symmetric Fokker-Planck type equations}, Communications on Stochastic Analysis 2 (2008), 153-175.

\bibitem{essay}
{\sc A.~Arnold, J.A.~Carrillo, L.~Desvillettes, J.~Dolbeault, A.~J\"ungel, C.~Lederman, P.~Markowich, G.~Toscani, C.~Villani}, {\em Entropies and equilibria of many-particle systems: an essay on recent research}, Monatsh. Math. 142 (2004), 35–43.

\bibitem{arnolddolbeault}
{\sc A.~Arnold, J.~Dolbeault}, {\em Refined convex Sobolev inequalities}, J. Funct.
Anal. 225 (2005), 337–351.

\bibitem{arnold}
{\sc A.~Arnold, P.~Markowich, G.~Toscani, A.~Unterreiter}, {\em On convex Sobolev inequalities and the rate of convergence to equilibrium for Fokker-Planck type equations},
Comm. Partial Differential Equations, 26 (2001), 43-100,.

\bibitem{arnoldunterreiter}
{\sc A.~Arnold, A.~Unterreiter}, {\em Entropy decay of discretized Fokker-Planck equations I: Temporal semidiscretization}, Computers and Mathematics with Applications, 46 (2003), 1683-1690.

\bibitem{bartels}
{\sc S.~Bartels}, {\em Error control and adaptivity for a variational model problem defined on functions of bounded variation}, Mathematics of Computation 84 (2015), 1217-1240.

\bibitem{BenamouCarlier}
{\sc J.-D.~ Benamou, G.~Carlier, M.~Cuturi, G.~Peyre, L.~Nenna}, {\em Iterative Bregman
projections for regularized transportation problems},  SIAM J. Sci. Comp. (2015), to appear.

\bibitem{benning}
{\sc M.~Benning, M.~Burger}, {\em Error estimates for general fidelities}, Electronic Transactions on Numerical Analysis 38 (2011), 44-68.

\bibitem{benninggroundstates}
{\sc M.~Benning, M.~Burger}, {\em Ground states and singular vectors of convex variational regularization methods}, Methods and Applications of Analysis 20 (2013) ,295 - 334.

\bibitem{bleyer}
{\sc I.R.~Bleyer, A.~Leitao}, {\em On Tikhonov functionals penalized by Bregman distances}, CUBO 11 (2008), 99-115.

\bibitem{bregman}
{\sc L.M. Bregman}, {\em The relaxation method for finding the common point of
  convex sets and its application to the solution of problems in convex
  programming}, USSR Comp. Math. Math. Phys., 7 (1967), 200--217.

\bibitem{bodineau}
{\sc T.~Bodineau, J.~Lebowitz, C.~Mouhot, C.~Villani}, {\em Lyapunov functionals for boundary-driven nonlinear drift–diffusion equations}, Nonlinearity 27 (2014), 2111.

\bibitem{schlake}
{\sc M.~Burger, M.~DiFrancesco, J.~F.~Pietschmann, B.~Schlake}, {\em Nonlinear cross-diffusion
with size exclusion}, SIAM J. Math. Anal. 42 (2010), 2842-2871.

\bibitem{adaptive1}
{\sc M.~Burger, M.~Moeller, M.~Benning, S.~Osher}, {\em An Adaptive inverse scale space method for compressed sensing}, Mathematics of Computation 82 (2013), 269-299. 

\bibitem{inversetvflow}
{\sc M.~Burger, K.~Frick, S.~Osher, O.~Scherzer}, {\em Inverse total variation flow}, Multiscale Modeling and Simulation
 6 (2007), 366-395.

\bibitem{spectraltv}
{\sc M.~Burger, L.~Eckardt, G.~Gilboa, M.~Moeller}, {\em Spectral representation of one-homogeneous functionals}, SSVM 2015, to appear.

\bibitem{BL14}
{\sc M.~Burger, F.~Lucka},
{\em Maximum-a-posteriori estimates in linear inverse problems with
  log-concave priors are proper bayes estimators}, Inverse Problems 30 (2014),  114004.

\bibitem{tvzoo1}
{\sc M.~Burger, S.~Osher}, {\em A guide to the TV zoo}, in Level Set and PDE
  -based Reconstruction Methods in Imaging, M.Burger and S.Osher, eds.,
  Springer, Berlin, 2013.

\bibitem{primaldualbregman}
{\sc C.~Brune, A.~Sawatzky, M.~Burger}, {\em Primal and dual {B}regman
  methods with application to optical nanoscopy}, Int. J. Comput. Vis., 92
  (2011), pp.~211--229.

\bibitem{gilboa}
{\sc M.~Burger, G.~Gilboa, S.~Osher, and J.~Xu}, {\em Nonlinear inverse scale
  space methods}, Comm. Math. Sci., 4 (2006), pp.~179--212.

\bibitem{BuOs04}
{\sc M.~Burger and S.~Osher}, {\em Convergence rates of convex variational regularization}, Inverse Problems 20 (2004), 1411--1421.

\bibitem{bur07}
{\sc M.~Burger, E.~Resmerita, and L.~He}, {\em Error estimation for bregman
  iterations and inverse scale space methods in image restoration}, Computing,
  81 (2007), pp.~109--135.

\bibitem{cai1}
{\sc J.~Cai, S.~Osher, and Z.~Shen}, {\em Linearized Bregman iterations for
  compressed sensing}, Math. Comp., 78 (2009), pp.~1515--1536.


\bibitem{CarlierJimenez}
{\sc G.~Carlier, C.~Jimenez}, {\em On Monge’s problem for Bregman-like cost functions}, Journal
of Convex Analysis  14  (2007), 647-655.

\bibitem{carrillo}
{\sc J.A.~Carrillo, A.~J\"ungel, P.~Markowich, G.~Toscani, A.~Unterreiter}, {\em Entropy dissipation methods for degenerate parabolic problems and generalized Sobolev inequalities}, Monatsh. Math. 133 (2001), 1–82

\bibitem{censorzenios}
{\sc Y.~Censor, S.~Zenios}, {\em Parallel Optimization: Theory, Algorithms, and Applications},  Oxford
University Press, 1998.

\bibitem{censorlent}
{\sc Y.~Censor, A.~Lent},  {\em An iterative row-action method for interval convex programming}, Journal of Optimization Theory and Applications, 34 (1981), 321—353.

\bibitem{chaventkunisch}
{\sc G.~Chavent and K.~Kunisch }, {\em Regularization of linear least squares problems by total bounded variation}. ESAIM: Control, Optimisation and Calculus of Variations 2 (1997) 359-376.

\bibitem{choi}
{\sc K.~Choi, B.~P.~Fahimian, T.~Li, T.~S.~ Suh, X.~Lei}, {\em Enhancement of four-dimensional cone-beam computed tomography by compressed sensing with Bregman iteration}, J. Xray Sci. Technol.  21 (2013), 177-192. 

\bibitem{markowichfellner}
{\sc M.~DiFrancesco, K.~Fellner, P.~Markowich}, {\em The entropy dissipation method for inhomogeneous reaction--diffusion systems}, Proc. Royal Soc. A 464 (2008) 3272-300.

\bibitem{dieningkreuzer}
{\sc L.~Diening, C.~Kreuzer}, {\em Linear convergence of an adaptive finite element method for the p-laplacian equation}, SIAM J. Numer. Anal. 46 (2008), 614-638.  

\bibitem{droniou}
{\sc J.~Droniou, J.L.~Vazquez}, {\em Noncoercive convection–diffusion elliptic problems with Neumann boundary conditions},  Calculus of Variations and Partial Differential Equations, 34 (2009), 413-434.

\bibitem{ekelandtemam}
{\sc I.Ekeland, R.Temam}, {\em Convex Analysis and Variational
Problems}, SIAM, Philadelphia, 1999.

\bibitem{EnHaNe96}
{\sc H.~Engl, M.~Hanke-Bourgeois, A.~Neubauer}, {\em Regularization of Inverse Problems},
Kluwer, Dordrecht, 1996.

\bibitem{flemming}
{\sc J.~Flemming}, {\em Theory and examples of variational regularization with non-metric fitting functionals}, Journal of Inverse and Ill-Posed Problems 18 (2010), 677-699.

\bibitem{splitbregman}
{\sc T.~Goldstein and S.~Osher}, {\em The split {B}regman method for $l_1$
  regularized problems}, SIAM J. on Imaging Sci., 2 (2008), pp.~323--343.

\bibitem{grasmair1}
{\sc M.Grasmair},{\em Generalized Bregman distances and convergence rates for non-convex regularization methods}, Inverse Problems 11 (2010), 115014.

\bibitem{grasmair2}
{\sc M.~Grasmair},{\em Variational inequalities and higher order convergence rates for Tikhonov regularisation on Banach spaces}, Journal of Inverse and Ill-Posed Problems  21 (2013), 379–394.

\bibitem{hein}
{\sc T.~Hein}, {\em Tikhonov regularization in Banach spaces—improved convergence rates results}, Inverse Problems, 25 (2009), 035002.

\bibitem{helin}
{\sc T.~Helin, M.~Burger}, {\em Maximum a posteriori probability estimates in infinite-dimensional Bayesian inverse problems}, Preprint arXiv:1412.5816 (2015).

\bibitem{juengelbounded}
{\sc A.~J\"ungel}, {\em The boundedness-by-entropy method for cross-diffusion systems}, to appear in Nonlinearity, 2015.

\bibitem{juengelentropie}
{\sc A.~J\"ungel, D.~Matthes}, {\em Entropiemethoden für nichtlineare partielle Differentialgleichungen}, Internat. Math. Nachrichten 209 (2008), 1-14.
 
\bibitem{stelzer}
{\sc A.~J\"ungel, I.~V.~Stelzer}, {\em Entropy structure of a cross-diffusion tumor-growth model},  To appear in Math. Mod. Meth. Appl. Sci., 2012.

\bibitem{kaipio}
{\sc J.~Kaipio, E.~Somersalo}, {\em Statistical and Computational Inverse Problems},
Springer, New York (2005).

\bibitem{Kiwiel}
{\sc K.~C.~Kiwiel}, {\em Proximal minimization methods with generalized Bregman functions}, SIAM J.
Control Optim., 35 (1997), pp. 1142–1168.
%

\bibitem{mielke}
{\sc A.~Mielke, J.~Haskovec, P.~A.~Markowich}, {\em On uniform decay of the entropy for reaction–diffusion systems}, Journal of Dynamics and Differential Equations (2013), 1-32. 

\bibitem{mielkerossi}
{\sc A.~Mielke, R.~Rossi, G.Savare}, {\em Nonsmooth analysis of doubly nonlinear evolution equations}, Calculus of Variations and Partial Differential Equations, 46 (2013), 253-310.

\bibitem{moel12}
{\sc M.~Moeller}, {\em Multiscale Methods for Polyhedral Regularizations and
  Applications in High Dimensional Imaging}, PhD thesis, University of
  Muenster, Germany, 2012.
	
\bibitem{adaptive2}
{\sc M.~Moeller, M.~Burger}, {\em Multiscale methods for polyhedral regularizations}, SIAM J. Optim. 23 (2013), 1424-1456. 
%
	
\bibitem{mueller}
{\sc J.~M\"uller, C.~Brune, A.~Sawatzky, T.~K\"osters, K.~P.~Sch\"afers, M.~Burger}, {\em Reconstruction of short time PET scans using Bregman iterations}, Nuclear Science Symposium and Medical Imaging Conference (NSS/MIC),  IEEE (2011), 2383-2385.
	
\bibitem{osh-bur-gol-xu-yin}
{\sc S.~Osher, M.~Burger, D.~Goldfarb, J.~Xu, and W.~Yin}, {\em An iterative
  regularization method for total variation-based image restoration}, SIAM
  Multiscale Model. Simul. 4 (2005), pp.~460--489.


\bibitem{otto}
{\sc F.~Otto}, {\em The geometry of dissipative evolution equations: the porous medium equation}, Comm. Partial Differential Equations, 26 (2001), 101-174.

\bibitem{poeschl}
{\sc C.~P\"oschl}, {\em An overview on convergence rates for Tikhonov regularization methods for non-linear operators, Journal of Inverse and Ill-posed Problems}, 17 (2009), 77-83.

\bibitem{reid}
{\sc M.~Reid}, {\em Meet the Bregman divergences}, http://mark.reid.name/blog/meet-the-bregman-divergences.html.

\bibitem{resmerita}{\sc E.~Resmerita}, {\em Regularization of ill-posed problems in Banach spaces: convergence rates}, Inverse Problems 21  (2005),  1303.

\bibitem{rof}
{\sc L.~I.~Rudin, S.~Osher, E.~Fatemi}, {\em Nonlinear total variation based noise removal algorithms}, Physica D: Nonlinear Phenomena 60 (1992),  259-268.

\bibitem{saintraymond}
{\sc L. Saint-Raymond}, {\em Convergence of solutions to the Boltzmann equation in the incompressible Euler limit}, Archive for Rational Mechanics and Analysis  166 (2003), 47-80.

\bibitem{scherzergroetsch}
{\sc O.~Scherzer, C.~Groetsch}, {\em Inverse scale space theory for inverse problems}, In: {\em Scale-Space and Morphology in Computer Vision}, Springer, Berlin, Heidelberg (2001), 317-325.  

\bibitem{wernerhohage}{\sc F.~Werner, T.~Hohage}, {\em Convergence rates in expectation for Tikhonov-type regularization of Inverse Problems with Poisson data}, Inverse Problems 28 (2012), 104004.

\bibitem{scalespace}
{\sc A.~P.~Witkin},{\em Scale-space filtering: A new approach to multi-scale description}, In: {\em Acoustics, Speech, and Signal Processing, IEEE International Conference on ICASSP'84}  9 (1984), 150-153.

\bibitem{Xu07}
{\sc J.~Xu and S.~Osher}, {\em Iterative regularization and nonlinear inverse
  scale space applied to wavelet-based denoising}, IEEE Trans. on Image
  Processing, 16 (2007), pp.~534--544.

\bibitem{yin}
{\sc W.~Yin}, {\em Analysis and generalizations of the linearized {B}regman
  method}, SIAM J. Imaging Sci., 3 (2010), pp.~856--877.

\bibitem{zamponi}
{\sc N.~Zamponi, A.~J\"ungel}, {\em Analysis of degenerate cross-diffusion population models with volume filling}, Preprint, TU Vienna, 2015.


\bibitem{zhang}
{\sc X.~Zhang, M.~Burger, X.~Bresson, and S.~Osher}, {\em Bregmanized nonlocal
  regularization for deconvolution and sparse reconstruction}, SIAM J. Imaging
  Sci., 3 (2010), pp.~253--276.
	
\end{thebibliography}
\end{document}